\documentclass[11pt]{amsart}

\usepackage{enumerate}
\usepackage{amsmath,amsthm,verbatim,amssymb,amsfonts,amscd,amsopn,amsxtra,graphicx,lmodern}
\usepackage{hyperref}
\usepackage{tikz-cd} 
\usepackage{graphics}
\graphicspath{ {./images/} }
\usepackage{mathrsfs}
\usepackage{relsize}
\usepackage{enumitem}
\usepackage[utf8]{inputenc}
\usepackage{csquotes}
\usepackage{amsthm}
\usepackage{thmtools}
\usepackage{mathtools}
\usepackage{microtype}
\usepackage{pdfrender,xcolor}
\usepackage[sort cites=true, backend=biber]{biblatex}
\usepackage{lua-visual-debug}
\usepackage{biblatex}

\tikzset{
	symbol/.style={
		draw=none,
		every to/.append style={
			edge node={node [sloped, allow upside down, auto=false]{$#1$}}}
	}
}

\begin{document}
	\pdfrender{StrokeColor=black,TextRenderingMode=2,LineWidth=0.2pt}	
	
	\title{On defectless unibranched simple extensions, complete distinguished chains and certain stability results}
	
	\author{Arpan Dutta}
	\email{arpandutta@iitbbs.ac.in}
	\address{Department of Mathematics, School of Basic Sciences, IIT Bhubaneswar, Argul,
		Odisha, India, 752050.}

	\author{Rumi Ghosh}
	\email{s23ma09008@iitbbs.ac.in}
	\address{Department of Mathematics, School of Basic Sciences, IIT Bhubaneswar, Argul,
		Odisha, India, 752050.}

	\def\NZQ{\mathbb}               
	\def\NN{{\NZQ N}}
	\def\QQ{{\NZQ Q}}
	\def\ZZ{{\NZQ Z}}
	\def\RR{{\NZQ R}}
	\def\CC{{\NZQ C}}
	\def\AA{{\NZQ A}}
	\def\BB{{\NZQ B}}
	\def\PP{{\NZQ P}}
	\def\FF{{\NZQ F}}
	\def\GG{{\NZQ G}}
	\def\HH{{\NZQ H}}
	\def\UU{{\NZQ U}}
	\def\P{\mathcal P}
	
	%
	%
	\let\union=\cup
	\let\sect=\cap
	\let\dirsum=\oplus
	\let\tensor=\otimes
	\let\iso=\cong
	\let\Union=\bigcup
	\let\Sect=\bigcap
	\let\Dirsum=\bigoplus
	\let\Tensor=\bigotimes
	
	\theoremstyle{plain}
	\newtheorem{Theorem}{Theorem}[section]
	\newtheorem{Lemma}[Theorem]{Lemma}
	\newtheorem{Corollary}[Theorem]{Corollary}
	\newtheorem{Proposition}[Theorem]{Proposition}
	\newtheorem{Problem}[Theorem]{}
	\newtheorem{Conjecture}[Theorem]{Conjecture}
	\newtheorem{Question}[Theorem]{Question}
	
	\theoremstyle{definition}
	\newtheorem{Example}[Theorem]{Example}
	\newtheorem{Examples}[Theorem]{Examples}
	\newtheorem{Definition}[Theorem]{Definition}
	
	\theoremstyle{remark}
	\newtheorem{Remark}[Theorem]{Remark}
	\newtheorem{Remarks}[Theorem]{Remarks}

	\newcommand{\n}{\par\noindent}
	\newcommand{\nn}{\par\vskip2pt\noindent}
	\newcommand{\sn}{\par\smallskip\noindent}
	\newcommand{\mn}{\par\medskip\noindent}
	\newcommand{\bn}{\par\bigskip\noindent}
	\newcommand{\pars}{\par\smallskip}
	\newcommand{\parm}{\par\medskip}
	\newcommand{\parb}{\par\bigskip}

	\let\epsilon=\varepsilon
	\let\phi=\varphi
	\let\kappa=\varkappa
	
	\newcommand{\trdeg}{\mbox{\rm trdeg}\,}
	\newcommand{\rr}{\mbox{\rm rat rk}\,}
	\newcommand{\sep}{\mathrm{sep}}
	\newcommand{\ac}{\mathrm{ac}}
	\newcommand{\ins}{\mathrm{ins}}
	\newcommand{\res}{\mathrm{res}}
	\newcommand{\Gal}{\mathrm{Gal}\,}
	\newcommand{\ch}{\operatorname{char}}
	\newcommand{\Aut}{\mathrm{Aut}\,}
	\newcommand{\kras}{\mathrm{kras}\,}
	\newcommand{\dist}{\mathrm{dist}\,}
	\newcommand{\ord}{\mathrm{ord}\,}
	\newcommand{\Div}{\mathrm{Div}\,}
	\newcommand{\Supp}{\mathrm{Supp}\,}
	\newcommand{\Spec}{\mathrm{Spec}\,}
	\newcommand{\height}{\mathrm{ht}\,}
	\newcommand{\rk}{\mathrm{rk}\,}
	\newcommand{\Diff}{\mathrm{Diff}\,}
	\newcommand{\Ram}{\mathrm{Ram}\,}
	\newcommand{\id}{\mathrm{id}\,}
	\newcommand{\lex}{\mathrm{lex}\,}
	\newcommand{\gr}{\mathrm{gr}\,}
	\newcommand{\init}{\mathrm{in}\,}

	\let\phi=\varphi
	\let\kappa=\varkappa
	
	\def \a {\alpha}
	\def \b {\beta}
	\def \s {\sigma}
	\def \d {\delta}
	\def \g {\gamma}
	\def \o {\omega}
	\def \l {\lambda}
	\def \th {\theta}
	\def \D {\Delta}
	\def \G {\Gamma}
	\def \O {\Omega}
	\def \L {\Lambda}
	%
	%
	\textwidth=15cm \textheight=22cm \topmargin=0.5cm
	\oddsidemargin=0.5cm \evensidemargin=0.5cm \pagestyle{plain}


	
	\date{\today}
	
	\maketitle
	

\begin{abstract}
	Let $(K,v)$ be a valued field. Take an extension of $v$ to a fixed algebraic closure $\overline{K}$ of $K$. In this paper we show that an element $a\in \overline{K}$ admits a complete distinguished chain over $K$ if and only if the extension $(K(a)|K,v)$ is unibranched and defectless. This characterization generalizes the known result in the henselian case. In particular, our result shows that if $a$ admits a complete distinguished chain over $K$, then it also admits one over the henselization $K^h$; however the converse may not be true. The main tool employed in our analysis is the stability of the $j$-invariant associated to a valuation transcendental extension under passage to the henselization.
	
	\pars We also explore the stability of defectless simple extensions in the following sense: let $(K(X)|K,w)$ be a valuation transcendental extension with a pair of definition $(b,\g)$. Assume that either $(K(b)|K,v)$ is a defectless extension, or that $f(X)$ is a key polynomial for $w$ over $K$, where $f(X)$ is the minimal polynomial of $b$ over $K$. We show that then the extension $(K(b,X)|K(X),w)$ is defectless. In particular, the extension $(K(b,X)|K(X),w)$ is always defectless whenever $(b,\g)$ is a minimal pair of definition for $w$ over $K$.  
\end{abstract}


\section{Introduction} Let $K$ be a field equipped with a Krull valuation $v$. The value group of $(K,v)$ will be denoted by $vK$ and the residue field by $Kv$. The value of an element $c\in K$ will be denoted by $vc$ and its residue by $cv$. Fix an algebraic closure $\overline{K}$ of $K$ and take an extension of $v$ to $\overline{K}$ which we again denote by $v$. Then for any algebraic extension $L|K$ we can talk of the \textbf{henselization} $L^h$ without any ambiguity.

\pars The notion of \textit{defect} is of singular importance in valuation theory. It plays a crucial role in some of the most fundamental problems arising in algebraic geometry and model theory. For example, it is known that the defect is the primary obstruction to achieving Local Uniformization in positive characteristic (cf. [\ref{Kuh vln model}]). We refer the reader to [\ref{Kuh defect}] for an extensive treatment of the defect. A finite extension $(L|K,v)$ is said to be \textbf{defectless} if there is equality in the Fundamental Inequality, which is equivalent to the extension $(L^h|K^h,v)$ being defectless. By the Lemma of Ostrowski, we then have the following relation for defectless extensions: 
\[  [L^h:K^h] = (vL:vK)[Lv:Kv]. \]

\pars A characterization of defectless simple extensions over \textit{henselian} valued fields has been provided in [\ref{KA SKK chains associated with elts henselian}] via the notion of complete distinguished chains. A pair $(b,a) \in \overline{K}\times\overline{K}$ is said to form a \textbf{distinguished pair over $K$} if the following conditions are satisfied:
\sn (DP1) $[K(b):K] > [K(a):K]$,
\n (DP2) $[K(z):K] < [K(b):K] \Longrightarrow v(b-a)\geq v(b-z)$, 
\n (DP3) $v(b-a) = v(b-z)\Longrightarrow [K(z):K] \geq [K(a):K]$. \\
In other words, $a$ is closest to $b$ among all the elements $z$ satisfying $[K(z):K] < [K(b):K]$; furthermore, $a$ has minimum degree among all such elements which are closest to $b$. In this case we define $\d(b,K):= v(b-a)$. Equivalently,
\[ \d(b,K):= \max \{ v(b-z) \mid [K(z):K] < [K(b):K] \}.  \]
An element $a\in \overline{K}$ is said to admit a \textbf{complete distinguished chain over $K$} if there is a chain $a_0(= a), a_1, \dotsc , a_n$ of elements in $\overline{K}$ such that $(a_i, a_{i+1})$ is a distinguished pair over $K$ for all $i$, and $a_n \in K$. In the setup of henselian valued fields, the existence of complete distinguished chains is provided by [\ref{KA SKK chains associated with elts henselian}, Theorem 1.2]:
\begin{Theorem}\label{Thm cdc Khanduja-Aghigh}
	Let $(K,v)$ be a henselian valued field. Then an element $a\in \overline{K}$ admits a complete distinguished chain over $K$ if and only if $(K(a)|K,v)$ is a defectless extension. 
\end{Theorem}

\pars Our primary goal in this present paper is to give a characterization of \textit{unibranched} defectless simple extensions over \textit{arbitrary} valued fields. A finite extension $(L|K,v)$ is said to be \textbf{unibranched} if there is a unique extension of $v$ from $K$ to $L$. By [\ref{Kuh max imm extns of valued fields}, Lemma 2.1], this is equivalent to $L$ being linearly disjoint to $K^h$ over $K$. Our central result is the following:
\begin{Theorem}\label{Thm central}
	An element $a\in\overline{K}$ admits a complete distinguished chain over $K$ if and only if $(K(a)|K,v)$ is a defectless and unibranched extension. 
\end{Theorem}
This extends Theorem \ref{Thm cdc Khanduja-Aghigh} to the setup of arbitrary valued fields. As a consequence of Theorem \ref{Thm central}, we observe that if $a$ admits a complete distinguished chain over $K$, then it also admits one over $K^h$. However, the converse is not true, since there can be defectless extensions which are not unibranched. A concrete example illustrating this fact is provided in Example \ref{Example}.

\pars The primary objects employed in our analysis of complete distinguished chains are the concepts of (minimal) pairs of definition and an invariant associated to valuation transcendental extensions. An extension $w$ of $v$ from $K$ to a rational function field $K(X)$ is said to be \textbf{valuation transcendental} if we have equality in the Abhyankar inequality, that is, 
\[ \dim_\QQ (\QQ\tensor_{\ZZ} wK(X)/vK) + \trdeg[K(X)w:Kv] = 1.   \]
We take an extension of $w$ to $\overline{K}(X)$ and denote it again by $w$. It is well-known (c.f. [\ref{AP sur une classe}], [\ref{APZ characterization of residual trans extns}], [\ref{APZ2 minimal pairs}], [\ref{Dutta min fields implicit const fields}], [\ref{Kuh value groups residue fields rational fn fields}]) that such an extension $w$ is completely characterized by a pair $(a,\g)\in \overline{K}\times w\overline{K}(X)$ in the following sense:
\[ w(X-z) = \min\{v(a-z),\g\} \text{ for all } z\in\overline{K}.  \]
We then say $(a,\g)$ is a \textbf{pair of definition for $w$} and write $w= v_{a,\g}$. A pair of definition for $w$ may not be unique. It has been observed in [\ref{AP sur une classe}, Proposition 3] that
\[ (b,\g^\prime) \text{ is also a pair of definition for $w$ if and only if } v(a-b)\geq \g=\g^\prime.  \] 
We say that $(a,\g)$ is a \textbf{minimal pair of definition for $w$ over $K$} if it has minimal degree over $K$ among all pairs of definition, that is, 
\[  v(a-b)\geq \g \Longrightarrow [K(b):K]\geq [K(a):K]. \]
The connection between minimal pairs of definition and distinguished pairs is captured by the following observation which is immediate and hence its proof is omitted:
\begin{Proposition}\label{Prop dist pair min pair}
	Assume that $(b,a)$ is a distinguished pair over $K$ and set $\g = \d(b,K)$. Then $(a,\g)$ is a minimal pair of definition for $v_{b,\g}$ over $K$. Moreover, if $\g^\prime$ is an element in an ordered abelian group containing $v\overline{K}$ such that $\g^\prime>\g$, then $(b,\g^\prime)$ is a minimal pair of definition for $v_{b,\g^\prime}$ over $K$.  
\end{Proposition} 
To a valuation transcendental extension $w$, we associate an invariant called the \textbf{$j$-invariant} in the following way. Take a pair of definition $(b,\g)$ of $w$. For any polynomial $f(X) = (X-z_1)\dotsc (X-z_n)\in \overline{K}[X]$, we define 
\[ j_w(f):= | \{z_i \mid v(b-z_i)\geq \g\} |,   \] 
where $|\O|$ denotes the cardinality of a set $\O$. When the valuation $w$ is tacitly understood, we will drop the suffix and simply write it as $j(f)$. The invariance of this $j$-invariant is established in [\ref{Dutta invariant of valn tr extns and connection with key pols}, Theorem 3.1]. Moreover, it has been observed in [\ref{Dutta invariant of valn tr extns and connection with key pols}, Theorem 3.3 and Proposition 3.5] that over $K[X]$, the $j$-invariant is completely characterized by $j(Q)$ whenever $Q(X)$ is the minimal polynomial of $a$ over $K$ for some \textit{minimal} pair of definition $(a,\g)$ of $w$ over $K$. In this present paper, we establish the \textit{stability} of the $j$-invariant when we pass to the henselization in Proposition \ref{Prop j(f) = j(f^h)}. We observe that given a pair of definition $(b,\g)$ of $w$ and the minimal polynomials $f(X)$ and $f^h(X)$ of $b$ over $K$ and $K^h$ respectively, then
\[  j(f) = j(f^h), \]
provided that $f$ is a key polynomial for $w$ over $K$. We direct the reader to Section \ref{Sec 2} for the definition of key polynomials. A necessary condition for an irreducible polynomial $Q(X)$ to be a key polynomial is that $j(Q)>0$ (cf. [\ref{Dutta invariant of valn tr extns and connection with key pols}, Section 4]). Thus a key polynomial over $K$ must be the minimal polynomial for some pair of definition over $K$. In particular, the conditions of Proposition \ref{Prop j(f) = j(f^h)} are always satisfied when we choose a \textit{minimal} pair of definition (cf. Proposition \ref{Prop j(f) = j(f^h) = [K(b)^h:IC_K]}). 

\pars The stability of the $j$-invariant plays a crucial role in the proof of the following lemma which shows that distinguished pairs over $K$ remain so over $K^h$. In particular, it is a key cog in the proof of Theorem \ref{Thm central}.

\begin{Lemma}\label{Lemma dist pair henselization}
	Let $(b,a)$ be a distinguished pair over $K$. Then $(b,a)$ is also a distinguished pair over $K^h$. Furthermore, the extension $(K(b)|K,v)$ is defectless (unibranched) if and only if the extension $(K(a)|K,v)$ is also defectless (unibranched).
\end{Lemma}

\pars Another significant consequence of the stability of the $j$-invariant is the observation that the extension $(K(a,X)|K(X),w)$ is defectless whenever $(a,\g)$ is a \textit{minimal} pair of definition for $w$ over $K$ (Corollary \ref{Coro K(a,X)|K(X) defectless}). This provides a very nice condition for the \textit{elimination of defect} of simple extensions. It is natural to inquire whether this observation holds true for any pair of definition. We obtain the following result in this direction:

\begin{Theorem}\label{Thm K(b,X)|K(b) is defectless}
	Let $(K(X)|K,w)$ be a valuation transcendental extension. Take an extension of $w$ to $\overline{K}(X)$ and take a pair of definition $(b,\g)$ for $w$. Take the minimal polynomial $f(X)$ of $b$ over $K$. Assume that either the extension $(K(b)|K,v)$ is defectless or that $f$ is a key polynomial for $w$ over $K$. Then $(K(b,X)|K(X),w)$ is defectless. 
\end{Theorem}
A primary ingredient in the proof of the above theorem is the stability of the defect of extensions generated by pairs of definitions coming from key polynomials over henselian fields (Proposition \ref{Prop d(K(a)|K) = d(K(b)|K)}).

\pars Another motivation for Theorem \ref{Thm K(b,X)|K(b) is defectless} comes from the Generalized Stability Theorem (cf. [\ref{Kuh gen stab thm}, Theorem 1.1]). A valued field $(K,v)$ is said to be defectless if every finite extension $(L|K,v)$ is defectless. The Generalized Stability Theorem asserts the stability of defectless fields under suitable extensions. In particular, it posits that if $(K,v)$ is a defectless field and $(K(X)|K,w)$ is a valuation transcendental extension then $(K(X), w)$ is also defectless. The following question is natural:
\begin{itemize}
	\item Let $(K(b)|K,v)$ be a defectless extension and $(K(X)|K,w)$ a valuation transcendental extension. Is the extension $(K(b,X)|K(X),w)$ also defectless? 
\end{itemize}  
Our analysis yields an affirmative answer to this problem under the assumption that $(b,\g)$ is a pair of definition for $w$ for some $\g\in w\overline{K}(X)$.

The conclusion of Theorem \ref{Thm K(b,X)|K(b) is defectless} fails to hold if we remove the starting assumptions (Example \ref{Example K(b,X)|K(X) defect}). On the other hand, the converse to Theorem \ref{Thm K(b,X)|K(b) is defectless} is also not true (Example \ref{Example K(b,X)|K(X) defectless}).


\section{Analysis of the $j$-invariant}\label{Sec 2}

\subsection{Associated graded algebra and key polynomials} Let $(K(X)|K,w)$ be a valuation transcendental extension. For any $\g\in wK(X)$, we define
\begin{align*}
	P_\g&:= \{ f\in K[X] \mid wf\geq \g \}, \\
	P_\g^+&:= \{ f\in K[X] \mid wf > \g \}.
\end{align*}
Then $P_\g$ is an additive subgroup of $K[X]$ and $P_\g^+$ is a subgroup of $P_\g$. We define 
\[ \gr_w (K[X]) := \Dirsum_{\g\in wK(X)} P_\g / P_\g^+.  \]
For any $f\in K[X]$ with $wf=\g$, the image of $f$ in $P_\g / P_\g^+$ is said to be the \textbf{initial form} of $f$, and will be denoted by $\init_w (f)$. Given $f,g\in K[X]$ with $wf=wg$, we observe that
\[  \init_w(f) + \init_w(g) = 
\begin{cases}
	\init_w(f+g) &\quad\text{if}\quad wf =wg = w(f+g),\\
	0 &\quad\text{if}\quad w(f+g) > wf=wg.\\
\end{cases}
\]
Thus, 
\[ \init_w(f) = \init_w(g) \Longleftrightarrow w(f-g)>wf=wg.  \]
We define a multiplication on $\gr_w(K[X])$ by setting
\[ \init_w(f)\init_w(g):= \init_w(fg).  \]
We can directly check that this makes $\gr_w(K[X])$ into an integral domain. We will call it the \textbf{associated graded ring}. Observe that $\gr_w(K[X])$ is a graded ring with $P_\g / P_\g^+$ being the $\g$-th homogeneous component. Thus $\init_w(f)$ is homogeneous for every $f\in K[X]$. It follows that whenever $\init_w(g)$ is divisible by $\init_w(f)$ where $f,g \in K[X]$, there exists some $h\in K[X]$ such that $\init_w(g) = \init_w(h)\init_w(f)$.

\pars A monic polynomial $Q(X)\in K[X]$ is said to be a \textbf{key polynomial for $w$ over $K$} if it satisfies the following conditions for any $f,g,h \in K[X]$:
\mn (KP1) $\init_w(Q) \mid \init_w(f) \Longrightarrow \deg (f) \geq \deg (Q)$,
\n (KP2) $\init_w(Q) \mid \init_w(gh) \Longrightarrow \init_w(Q)\mid \init_w(g)$ or $\init_w(Q)\mid \init_w(h)$.\\
This notion of key polynomials was first introduced by Mac Lane [\ref{Mac Lane key pols}, \ref{Mac Lane key pols prime ideals}] and later extended by Vaqui\'{e} [\ref{Vaquie key pols}]. 

\pars Another notion of key polynomials, called the abstract key polynomials, was introduced by Spivakovsky and Novacoski in [\ref{Nova Spiva key pol pseudo convergent}]. Take an extension of $w$ to $\overline{K}(X)$. For any $f(X)\in K[X]$, we define
\[ \d_w(f):= \max \{ w(X-z) \mid \text{$z$ is a root of $f$} \}.  \]
It has been shown in [\ref{Novacoski key poly and min pairs}] that this value is independent of the choice of the extension. We say that a monic polynomial $Q(X)$ is an \textbf{abstract key polynomial for $w$ over $K$} if for any $f(X)\in K[X]$ we have
\[ \deg f < \deg Q \Longrightarrow \d_w(f) < \d_w(Q).   \]

\subsection{Associated value transcendental extension}\label{Section asso value tr extn} Let $(K(X)|K,w)$ be a valuation transcendental extension. Take an extension of $w$ to $\overline{K}(X)$ and take a minimal pair of definition $(a,\g)$ for $w$ over $K$. If $\g\notin v\overline{K}$ we set $\tilde{w}:= w$. If $\g\in v\overline{K}$, consider the ordered abelian group $G:= (v\overline{K}\dirsum \ZZ)_{\lex}$ and embed $v\overline{K}$ into $G$ by $\a \longmapsto (\a,0)$ for all $\a\in v\overline{K}$. Set $\G:= (\g, -1)$ and $\tilde{w}:= v_{a,\G}$. Then 
\[  \a\geq \g \text{ if and only if } \a\geq\G \text{ for all } \a\in v\overline{K}.   \]
As a consequence, $(a,\G)$ is a minimal pair of definition for $\tilde{w}$ over $K$ and $j_w(g) = j_{\tilde{w}}(g)$ for all $g(X)\in K[X]$. 

\begin{Remark}
	The valuation $\tilde{w}$ constructed above is referred to as the \textbf{associated value transcendental extension} in [\ref{Dutta invariant of valn tr extns and connection with key pols}, \ref{Dutta min pairs inertia ram deg impl const fields}] and plays an important role in obtaining bounds of the $j$-invariant (cf. [\ref{Dutta min fields implicit const fields}, \ref{Dutta min pairs inertia ram deg impl const fields}]).
\end{Remark}

\begin{Lemma}\label{Lemma deg f/ deg Q = j(f)/j(Q)}
	Let $(K(X)|K,w)$ be a valuation transcendental extension. Take an extension of $w$ to $\overline{K}(X)$ and take a minimal pair of definition $(a,\g)$ for $w$ over $K$. Take the minimal polynomial $Q(X)$ of $a$ over $K$ and any polynomial $f(X)\in K[X]$. Then 
	\[ \dfrac{j(f)}{j(Q)} \leq \dfrac{\deg f}{\deg Q}.  \]
	The above inequality is an equality whenever $f$ is a key polynomial for $w$ over $K$. If $(K,v)$ is henselian, then the inequality is an equality whenever $f$ is irreducible and $j(f)>0$.
\end{Lemma}

\begin{proof}
	The first two assertions follow from [\ref{Dutta invariant of valn tr extns and connection with key pols}, Remark 3.8]. It is thus enough to show the final assertion. We can assume that $f$ is monic. Since $j(f)>0$, we take a root $b$ of $f$ such that $v(a-b)\geq\g$. The triangle inequality and the minimality of $(a,\g)$ imply that $w(X-z)\leq v(b-z)$ for all $z\in\overline{K}$, and equality holds whenever $[K(z):K]< \deg Q$. It follows that
	\begin{equation}\label{Eqn 1}
		wg \leq vg(b) \text{ for all } g(X)\in K[X], \text{ and  equality holds whenever } \deg g < \deg Q.
	\end{equation}    
Take the minimal polynomial $m(X):= X^d + c_{d-1}X^{d-1} + \dotsc + c_1 X + c_0$ of $Q(b)$ over $K$. For any $i<d$, applying Vieta's formulas and the triangle inequality, we obtain that
\[  vc_i \geq v(\s_1 Q(b) \dotsc \s_{d-i} Q(b)) \text{ for some } \s_1, \dotsc , \s_{d-i} \in \Gal(\overline{K}|K).  \]
Since $(K,v)$ is henselian, employing (\ref{Eqn 1}) we observe that $vc_i \geq (d-i) vQ(b) \geq (d-i) \deg Q$ and hence
\begin{equation}\label{Eqn 2}
	w(c_i Q^i) \geq wQ^d \text{ for all } i< d.
\end{equation}
Set $M(X):= m(Q(X)) = Q^d + c_{d-1}Q^{d-1} + \dotsc + c_1 Q + c_0$. Then $wM = \min\{ wQ^d, w(c_i Q^i) \}$ by [\ref{Novacoski key poly and min pairs}, Theorem 1.1]. It follows from (\ref{Eqn 2}) that 
\begin{equation}\label{Eqn 3}
	wM = wQ^d. 
\end{equation}
Observe that $b$ is a root of $M$. The irreducibility of $f$ implies that it is the minimal polynomial of $b$ over $K$ and hence divides $M$ over $K$. We can thus express $M=fg$ for some $g\in K[X]$. We have unique expressions $f = \sum_{i=0}^{r} f_i Q^i$ and $g = \sum_{j=0}^{s} g_j Q^j$ where $f_i, g_j \in K[X]$ such that $\deg f_i, \deg g_j < \deg Q$. Employing [\ref{Novacoski key poly and min pairs}, Theorem 1.1] again we can take $i,j$ such that $wf = w(f_i Q^i)$ and $wg = w(g_jQ^j)$. It follows that
\[  dwQ = wM = w(fg) = w(f_ig_j) + (i+j)wQ.  \]
Observe that the value of the $j$-invariant and the choice of the minimal pair remains unchanged if we consider the associated value transcendental extension. We can thus assume that $\g\notin v\overline{K}$ without any loss of generality. In particular, this implies that $wQ$ is torsion-free modulo $vK$. Since $w(f_ig_j) \in v\overline{K}$ by (\ref{Eqn 1}), we conclude that 
\[  d=i+j. \]
Since $d\deg Q = \deg M  = \deg fg  = \deg (f_r g_s Q^{r+s}) \geq (i+j) \deg Q$, we conclude that $i=r$ and $f_r = 1$. In other words, we have that $f = Q^r + \sum_{i=0}^{r-1} f_i Q^i$ and $wf = wQ^r$. The assertion is now a direct consequence of [\ref{Dutta invariant of valn tr extns and connection with key pols}, Proposition 3.4]. 
\end{proof}

\subsection{Stability of the $j$-invariant}

\begin{Proposition}\label{Prop j(f) = j(f^h)}
	Let $(K(X)|K,w)$ be a valuation transcendental extension. Take an extension of $w$ to $\overline{K}(X)$ and take a pair of definition $(b,\g)$ for $w$. Take the minimal polynomials $f(X)$ and $f^h(X)$ of $b$ over $K$ and $K^h$ respectively. Assume that $f(X)$ is a key polynomial for $w$ over $K$. Then
	\[  j(f) = j(f^h). \] 
	Moreover, $f^h$ is a key polynomial for $w$ over $K^h$.
\end{Proposition}

\begin{proof}
	Observe that $j(f^h) \leq j(f)$ by definition. Since the canonical injection $\gr_w(K[X]) \subseteq  \gr_w(K^h[X])$ is actually an isomorphism by [\ref{Nart Novacoski defect formula}, Theorem 1.2], we have some $g(X)\in K[X]$ such that $\init_w(g) = \init_w(f^h)$. Then $j(g) = j(f^h)$ by [\ref{Dutta invariant of valn tr extns and connection with key pols}, Corollary 3.6]. Observe that $\init_w(f^h)$ divides $\init_w(f)$ in $\gr_w(K^h[X])$. Employing [\ref{Nart Novacoski defect formula}, Theorem 1.2] again, we obtain a polynomial $g_1(X)\in K[X]$ such that 
	\begin{equation}\label{Eqn 4}
		\init_w(f) = \init_w(g)\init_w(g_1).
	\end{equation}
	Then $j(f) = j(g)+j(g_1)$ by [\ref{Dutta invariant of valn tr extns and connection with key pols}, Proposition 3.2]. Since $j(g) = j(f^h)>0$ we conclude that $j(f) > j(g_1)$. By [\ref{Dutta invariant of valn tr extns and connection with key pols}, Corollary 3.6], $\init_w(f)$ does not divide $\init_w(g_1)$. Since $f$ is a key polynomial for $w$ over $K$, we conclude from (\ref{Eqn 4}) that $\init_w(f)$ divides $\init_w(g)$. Thus $j(f)\leq j(g) = j(f^h)$ and hence
	\[  j(f) = j(f^h). \] 
	Since $\init_w(f^h)$ divides $\init_w(f)$, we obtain from [\ref{Dutta invariant of valn tr extns and connection with key pols}, Theorem 3.7] that they are associates in $\gr_w(K^h[X])$. 
	
	\pars It remains to show that $f^h$ is a key polynomial for $w$ over $K^h$. Let $g(X)\in K^h[X]$ such that $\init_w(f^h)$ divides $\init_w(g)$. Then $j(g) \geq j(f^h)$ by [\ref{Dutta invariant of valn tr extns and connection with key pols}, Corollary 3.6]. As a consequence of Lemma \ref{Lemma deg f/ deg Q = j(f)/j(Q)}, we obtain that $\deg g \geq \deg f^h$, that is $f^h$ satisfies (KP1). The fact that $f^h$ also satisfies (KP2) follows from [\ref{Nart Novacoski defect formula}, Theorem 1.2], the assumption that $f$ is a key polynomial for $w$ over $K$ and the observation that $\init_w(f)$ and $\init_w(f^h)$ are associates.   
\end{proof}

In a slightly different form, the above result appears also in [\ref{Nart Novacoski defect formula}, Proposition 5.6]. This next result can also be found in [\ref{Nart Novacoski defect formula}, Theorem 5.16]. We present an alternate proof here. 

\begin{Corollary}\label{Coro min pair henselization}
	Let $(K(X)|K,w)$ be a valuation transcendental extension. Take an extension of $w$ to $\overline{K}(X)$ and take a minimal pair of definition $(a,\g)$ for $w$ over $K$. Then $(a,\g)$ is also a minimal pair of definition for $w$ over $K^h$.
\end{Corollary}

\begin{proof}
	Take a minimal pair of definition $(b,\g)$ for $w$ over $K^h$ and take the minimal polynomial $f^h$ of $b$ over $K^h$. As observed in Section \ref{Section asso value tr extn} we can assume that $\g \notin v\overline{K}$. It follows from [\ref{Dutta min fields implicit const fields}, Remark 3.3] that
	\[  vK(a) \dirsum \ZZ j(Q)\g = wK(X) = wK^h(X) = vK^h(b) \dirsum \ZZ j(f^h)\g.  \]
	Since $\g$ is torsion-free modulo $vK$, we conclude that $j(Q) = j(f^h)$. Take the minimal polynomial $Q^h$ of $a$ over $K^h$. Employing Proposition \ref{Prop j(f) = j(f^h)} and Lemma \ref{Lemma deg f/ deg Q = j(f)/j(Q)} we conclude that $\deg Q^h = \deg f^h$. The assertion now follows from the minimality of $(b,\g)$ over $K^h$. 
\end{proof}

The following description of the $j$-invariant is very useful when analysing the stability of defectless simple extensions (cf. Theorem \ref{Thm K(b,X)|K(b) is defectless}).

\begin{Proposition}\label{Prop j(f) = j(f^h) = [K(b)^h:IC_K]}
	Let notations and assumptions be as in Proposition \ref{Prop j(f) = j(f^h)}. Take an extension of $w$ to $\overline{K(X)}$. Then
	\[  j(f) = j(f^h) = [K(b,X)^h: K(X)^h].     \]
	In particular, this assertion is always true whenever $(b,\g)$ is a minimal pair of definition. 
\end{Proposition}

\begin{proof}
	We first assume that $(b,\g)$ is a minimal pair of definition for $w$ over $K$. Then $f$ is a key polynomial by [\ref{Dutta invariant of valn tr extns and connection with key pols}, Theorem 3.10(1)] and hence $j(f) = j(f^h)$ by Proposition \ref{Prop j(f) = j(f^h)}. From [\ref{Dutta min pairs inertia ram deg impl const fields}, Theorem 1.1] we obtain that
		\[  j(f) \leq [K(b,X)^h:K(X)^h].  \]
		Take a $K(X)^h$-conjugate $b^\prime$ of $b$. Then $b^\prime = \s b$ for some $\s$ in the decomposition group $G^d(\overline{K(X)}|K(X),w)$. It follows that
		\[ w(X-b^\prime) =  (w\circ \s) (X-b) = w(X-b) = \g. \]
		From the triangle inequality we obtain $v(b-b^\prime)\geq\g$ and as a consequence $[K(b,X)^h:K(X)^h]\leq j(f)$. The assertion now follows. 
		
		\pars We now assume that $(b,\g)$ is an arbitrary pair of definition for $w$. Take a minimal pair of definition $(a,\g)$ for $w$ over $K$ and take the minimal polynomials $Q$ and $Q^h$ of $a$ over $K$ and $K^h$. For any algebraic extension $L|K$, denote by $IC_L$ the relative algebraic closure of $K$ in $L(X)^h$. Observe that
		\[ K^h \subseteq IC_K = IC_{K^h}.   \]
		It has been observed in [\ref{Dutta min fields implicit const fields}, Lemma 5.1] that $IC_K \subseteq K(z)^h$ for any pair of definition $(z,\g)$. As a consequence, $IC_K(z) = K(z)^h$. Since $IC_K$ is relatively algebraically closed in $K(X)^h$, we obtain that
		\begin{equation}\label{Eqn 5}
			[K(z,X)^h:K(X)^h] = [K(z)^h:IC_K] \text{ for any pair of definition } (z,\g).
		\end{equation}
		In particular, it follows from our prior observations that
		\begin{equation}\label{Eqn 6}
			j(Q^h) = [K(a)^h:IC_K].
		\end{equation}
		Lemma \ref{Lemma deg f/ deg Q = j(f)/j(Q)} yields that
		\[ \dfrac{\deg f^h}{\deg Q^h} = \dfrac{j(f^h)}{j(Q^h)}.  \]
		In light of (\ref{Eqn 6}), we then conclude that $j(f^h) = [K(b)^h:IC_K]$. The proposition now follows from (\ref{Eqn 5}) and Proposition \ref{Prop j(f) = j(f^h)}.
\end{proof}

\begin{Remark}
	The field $IC_K$ is referred to as the \textbf{implicit constant field} and was introduced by Kuhlmann in [\ref{Kuh value groups residue fields rational fn fields}] to study extensions of valuations to rational function fields with prescribed value groups and residue fields. 
\end{Remark}

The following corollary is immediate in light of [\ref{Dutta min pairs inertia ram deg impl const fields}, Theorem 1.1]:

\begin{Corollary}\label{Coro K(a,X)|K(X) defectless}
	Let notations and assumptions be as in Corollary \ref{Coro min pair henselization}. Then $(K(a,X)|K(X),w)$ is a defectless extension. 
\end{Corollary}

We can relax the conditions of Proposition \ref{Prop j(f) = j(f^h) = [K(b)^h:IC_K]} even further over henselian fields:

\begin{Corollary}\label{Coro j(f) = [K(b):IC_K] when K=K^h}
	Let $(K(X)|K,w)$ be a valuation transcendental extension. Take an extension of $w$ to $\overline{K(X)}$ and take a pair of definition $(b,\g)$ for $w$ over $K$. Take the minimal polynomial $f(X)$ of $b$ over $K$. Assume that $(K,v)$ is henselian. Then
	\[  j(f) = [K(b,X)^h:K(X)^h].  \]
\end{Corollary}

\begin{proof}
	Take a minimal pair of definition $(a,\g)$ for $w$ over $K$ and the minimal polynomial $Q$ of $a$ over $K$. Then $\deg f / \deg Q = j(f)/j(Q)$ by Lemma \ref{Lemma deg f/ deg Q = j(f)/j(Q)}. We have $j(Q) = [K(a):IC_K]$ by Proposition \ref{Prop j(f) = j(f^h) = [K(b)^h:IC_K]}. As a consequence, we have that $j(f) = [K(b):IC_K]$. The assertion now follows. 
\end{proof}


\section{Analysis of distinguished pairs}

\begin{Lemma}\label{Lemma dist pair key pol}
	Let $(b,a)$ be a distinguished pair over $K$ and set $\g:= v(b-a)$. Let $f(X)$ be the minimal polynomial of $b$ over $K$. Then $f$ is a key polynomial for $w:= v_{b,\g}$ over $K$.
\end{Lemma}

\begin{proof}
 Take any $\g_1\in v\overline{K}$ such that $\g_1>\g$ and set $w_1:= v_{b,\g_1}$. Then $(b,\g_1)$ is a minimal pair of definition for $w_1$ over $K$. Hence $f$ is an abstract key polynomial for $w_1$ over $K$ by [\ref{Novacoski key poly and min pairs}, Theorem 1.1]. Take the minimal polynomial $Q(X)$ of $a$ over $K$. From the given conditions we can further observe that $Q$ is an abstract key polynomial for $w_1$ over $K$, moreover, $f$ is an abstract key polynomial of minimal degree such that $\d_{w_1}(f) > \d_{w_1}(Q)$. The assertion now follows from [\ref{Decaup Spiva Mahboub ABKP comparison MVKP}, Theorem 26]. 
\end{proof}

\subsection{Proof of Lemma \ref{Lemma dist pair henselization}}

\begin{proof}
	Denote by $f$ and $Q$ the minimal polynomials of $b$ and $a$ over $K$, and those over $K^h$ by $f^h$ and $Q^h$. Set $\g:= v(b-a)$ and $w:= v_{b,\g}$. Then $(a,\g)$ is a minimal pair of definition for $w$ over $K$, hence also over $K^h$ by Corollary \ref{Coro min pair henselization}. It follows from Lemma \ref{Lemma dist pair key pol} that $f$ is a key polynomial for $w$ over $K$. Then $f^h$ is a key polynomial for $w$ over $K^h$ by Proposition \ref{Prop j(f) = j(f^h)}. As a consequence, employing Proposition \ref{Prop j(f) = j(f^h)} and Lemma \ref{Lemma deg f/ deg Q = j(f)/j(Q)} we obtain that
	\begin{equation}\label{Eqn 7}
		\dfrac{\deg f}{\deg Q} = \dfrac{j(f)}{j(Q)} = \dfrac{j(f^h)}{j(Q^h)} = \dfrac{\deg f^h}{\deg Q^h}.
	\end{equation}
	 Since $(b,a)$ is a distinguished pair over $K$, we conclude that
	\begin{equation}\label{Eqn 8}
		[K^h(a):K^h] < [K^h(b):K^h].
	\end{equation} 
Take $z\in\overline{K}$ such that $\g_1:= v(b-z) > \g$. Then $(b,\g_1)$ is a minimal pair of definition for $w_1:= v_{b,\g_1}$ over $K$ and hence also over $K^h$ by Corollary \ref{Coro min pair henselization}. Thus $[K^h(b): K^h] \leq [K^h(z):K^h]$ whenever $v(b-z)> \g$. From (\ref{Eqn 8}) we conclude that
\[ \g= \d(b,K^h).  \]
The first assertion now follows from the minimality of $(a,\g)$ over $K^h$. 

\pars We now assume that $(K(b)|K,v)$ is defectless. Then $(K^h(b)|K^h,v)$ is defectless and hence $b$ admits a complete distinguished chain over $K^h$ by Theorem \ref{Thm cdc Khanduja-Aghigh}. Take such a chain $b, b_1, b_2, \dotsc , b_n$. Since $(b,a)$ is a distinguished pair over $K^h$, we observe that $b,a,b_2, \dotsc , b_n$ is also a complete distinguished chain over $K^h$. Then $a,b_2, \dotsc , b_n$ forms a complete distinguished chain of $a$ over $K^h$ and hence $(K(a)|K,v)$ is defectless by Theorem \ref{Thm cdc Khanduja-Aghigh}. Conversely, the fact that $(b,a)$ is a distinguished pair over $K^h$ implies that a complete distinguished chain of $a$ over $K^h$ extends to one of $b$. We thus have the reverse implication. 

\pars Finally, it follows from (\ref{Eqn 7}) that $f = f^h$ if and only if $Q=Q^h$. Employing [\ref{Kuh max imm extns of valued fields}, Lemma 2.1], we then observe that $(K(b)|K,v)$ is unibranched if and only if $(K(a)|K,v)$ is also unibranched.
\end{proof}

\subsection{Proof of Theorem \ref{Thm central}}

\begin{proof}
	We first assume that $a$ has a complete distinguished chain over $K$. Take such a chain $a, a_1, \dotsc , a_n$. By repeated applications of Lemma \ref{Lemma dist pair henselization} we obtain that it is also a complete distinguished chain over $K^h$. Hence $(K(a)|K,v)$ is defectless by Theorem \ref{Thm cdc Khanduja-Aghigh}. Since $a_n \in K$, $(K(a_n)|K,v)$ is unibranched. Repeated implementations of Lemma \ref{Lemma dist pair henselization} then yield that $(K(a)|K,v)$ is unibranched.
	
	\pars Conversely assume that $(K(a)|K,v)$ is defectless and unibranched. Then $a$ admits a complete distinguished chain $a, a_1, \dotsc , a_n$ over $K^h$ by Theorem \ref{Thm cdc Khanduja-Aghigh}. Set $\g:= v(a-a_1)$ and $w:= v_{a,\g}$. Then $(a_1, \g)$ is a minimal pair of definition for $w$ over $K^h$. Take a minimal pair or definition $(b,\g)$ for $w$ over $K$. By Corollary \ref{Coro min pair henselization}, $(b,\g)$ is also a minimal pair of definition for $w$ over $K^h$. As a consequence, 
	\begin{equation}\label{Eqn 9}
		[K^h(b):K^h] = [K^h(a_1):K^h] < [K^h(a):K^h] = [K(a):K],
	\end{equation}
	where the last equality follows from [\ref{Kuh max imm extns of valued fields}, Lemma 2.1]. Since $\d(a,K^h) = \g$ and $v(a-b)\geq \g$, we conclude that
	\[  v(a-b) = \g. \]
	It follows that
	\begin{equation}\label{Eqn 10}
		(a,b) \text{ is a distinguished pair over } K^h. 
	\end{equation}   
The minimality of $(b,\g)$ over $K$ implies that $[K(b):K] \leq [K(a):K]$. If we have equality, then $(a,\g)$ would also form a minimal pair of definition of $w$ over $K$ and hence over $K^h$ by Corollary \ref{Coro min pair henselization}. However, this would contradict (\ref{Eqn 9}). It follows that 
\begin{equation}\label{Eqn 11}
	[K(b):K] < [K(a):K]. 
\end{equation} 
Take $z\in\overline{K}$ such that $v(a-z) > \g = \d(a,K^h)$. Then we have the following chain of relations:
\[ [K(z):K] \geq [K^h(z):K^h] \geq [K^h(a):K^h] = [K(a):K].   \]
Consequently, from (\ref{Eqn 11}) we have that
\[ \d(a,K) = \g. \]
Employing the minimality of $(b,\g)$ over $K$ we conclude that
\[ (a,b)\text{ is a distinguished pair over } K.   \]
It then follows from Lemma \ref{Lemma dist pair henselization} that $(K(b)|K,v)$ is also defectless and unibranched. In light of (\ref{Eqn 10}) we observe that $a,b,a_2, \dotsc, a_n$ forms a complete distinguished chain of $a$ over $K^h$. Hence $b, a_2, \dotsc , a_n$ forms a complete distinguished chain of $b$ over $K^h$. The assertion now follows by induction.    
\end{proof}

The following corollary is immediate: 

\begin{Corollary}\label{Coro cdc over K implies over K^h}
	Assume that $a\in\overline{K}$ admits a complete distinguished chain over $K$. Then $a$ also admits a complete distinguished chain over $K^h$. 
\end{Corollary}

\begin{Proposition}
	Assume that $a\in\overline{K}$ admits a complete distinguished chain $a_0 (=a), a_1, \dotsc , a_n$ over $K$. Take the minimal polynomials $Q_i$ of $a_i$ over $K$. Then 
	\[ \deg Q_{i+1} \mid \deg Q_i \text{ for all } i\geq 0.  \]
\end{Proposition}

\begin{proof}
	Take any $0\leq i < n$ and consider $w:= v_{a_i,\g }$ where $\g:= v(a_i - a_{i+1}) = \d(a_i, K)$. Then $Q_i$ is a key polynomial for $w$ over $K$ by Lemma \ref{Lemma dist pair key pol}. Observe that $(a_{i+1}, \g)$ is a minimal pair of definition for $w$ over $K$. The assertion now follows from [\ref{Dutta invariant of valn tr extns and connection with key pols}, Remark 3.8].
\end{proof}

The converse to Corollary \ref{Coro cdc over K implies over K^h} is not true since there can be defectless extensions which are not unibranched. For example take any non-henselian field $(K,v)$ and take $b\in K^h\setminus K$. Then $(K(b)|K,v)$ is such an extension. A more involved example is provided underneath.

\begin{Example}\label{Example} 
	Let $k$ be a field with $\ch k = p >0$ such that $k$ is not Artin-Schreier closed, that is, $k$ admits irreducible Artin-Schreier polynomials. Denote by $(K,v)$ the valued field $k(t)$ equipped with the $t$-adic valuation. Then $\widehat{K}:= k((t))$ is the completion of $K$. Set
	\[ a:= \sum_{i=0}^{\infty} t^{p^i} \in \widehat{K}\setminus K.  \]
	Observe that $a^p -a + t = 0$. Then $a\in K^h = \widehat{K}\sect \overline{K}$. The fact that $a\notin K$ implies that the Artin-Schreier polynomial $X^p-X+t$ is irreducible over $K$ and hence 
	\[ [K(a):K] =p.  \]
	Let $X^p - X -c \in k[X]$ be irreducible over $k$. Take $b\in \overline{K}$ such that
	\[ b^p - b = c-a.  \]
	The fact that $v(c-a)=0$ implies that $vb =0$. Observe that $bv$ is a root of $X^p - X - c$. It follows that
	\[  [K(b)v: Kv] \geq [Kv(bv): Kv] =[k(bv) : k] = p.  \]
	Since $[K^h(b) : K^h]\leq p$, the Fundamental Inequality yields that
	\[ [K^h(b) : K^h] = p = [K(b)v: Kv].  \]
	In particular, $(K(b)|K,v)$ is a defectless extension. Observe that $a\in K(b)$. As a consequence we have that $[K(b) : K(a)] = p$ and hence $[K(b):K]= p^2$. Thus $(K(b)|K,v)$ is not unibranched by [\ref{Kuh max imm extns of valued fields}, Lemma 2.1].

	\parm We can also directly observe that $b$ admits a complete distinguished chain over $K^h$ but not over $K$. First of all, if $v(b-z)>0$ for some $z\in K^h$, then $zv = bv$ and hence $zv$ is a root of $X^p - X-c$. But this would imply that $X^p-X-c$ splits over $Kv = k$ which would contradict our starting assumption. It follows that $v(b-z)\leq 0$ for all $z\in K^h$. Since $vb =0$ we conclude that $\d(b,K^h) = 0$ and $b,0$ forms a complete distinguished chain of $b$ over $K^h$.   
	
	\pars We will now illustrate the fact that $b$ does not admit a complete distinguished chain over $K$. For all $n\in\NN$, take $b_n \in \overline{K}$ such that 
	\[ b_n^p - b_n = c-t- t^p - \dotsc - t^{p^n}.  \]
	Then,
	\begin{equation}\label{Eqn 12}
		(b_n - b)^p - (b_n - b) = a - t- t^p - \dotsc - t^{p^n} = \sum_{i=n+1}^{\infty} t^{p^i}.
	\end{equation}
	Observe that $b_n v$ is also a root of the Artin-Schreier polynomial $X^p - X - c$ and hence $b_n v = bv - \zeta$ for some $\zeta \in \FF_p$. It follows that $v(b_n - b + \zeta)>0$. Since $(b_n - b + \zeta)^p - (b_n - b + \zeta) = (b_n-b)^p - (b_n-b) $, we infer from (\ref{Eqn 12}) that
	\[ v(b - b_n - \zeta) = p^{n+1}.   \]
	Observe that $[K(b_n+\zeta):K] = [K(b_n):K] = p < [K(b):K]$. Thus for any $n\in\NN$ we can find $b^\prime_n$ such that $[K(b^\prime_n) : K] < [K(b):K]$ and $v(b-b^\prime_n) = p^{n+1}$. We conclude that there does not exist any $b^\prime\in\overline{K}$ such that $(b,b^\prime)$ forms a distinguished pair over $K$. In particular, $b$ does not admit a complete distinguished chain over $K$. 
\end{Example}

This next example illustrates that the converse to Proposition \ref{Prop j(f) = j(f^h)} is not true: 

\begin{Example}
	Let $(K,v)$ and $a$ be as in Example \ref{Example}. Take the valuation $w:= v_{0,\g}$ for any real number $0<\g<1$. The fact that $va=1>\g$ implies that $(a,\g)$ is a pair of definition for $w$. The minimal polynomial of $a$ over $K$ is given by the Artin-Schreier polynomial $f(X):= X^p-X+t$. Then $v(a-a^\prime) = 0<\g$ for any $K$-conjugate $a^\prime$ of $a$ which is distinct from $a$. Moreover, the polynomial $f(X)$ is separable. It follows that $j(f) = 1$. Since $a\in K^h$, we have that $f^h = X-a$ and hence $j(f^h)=1$ as well. Observe that 
	\[ w(f) = w(X^p - X + t) = \min \{ p\g, \g, 1 \} = \g < w(X^p+t).  \]
	Thus $\init_w (f) = \init_w (X)$. Since $\deg f > \deg X$, it follows that $f$ is not a key polynomial for $w$ over $K$.
\end{Example}


\section{Defect and pairs of definition}

\subsection{Stability of defect over henselian fields}

\begin{Definition}
	Let $(K(X)|K,w)$ be a valuation transcendental extension. Take a key polynomial $f$ for $w$ over $K$ and take $\b$ in some ordered abelian group containing $wK(X)$ such that $\b> wf$. We define the map $w^\prime : K[X] \longrightarrow wK(X)+\ZZ\b$ by setting $w^\prime g:= \min \{ wg_i + i\b \}$, where $g = \sum g_i f^i$ is the unique expansion with $\deg g_i < \deg f$. Extending $w^\prime$ canonically to $K(X)$ defines a valuation on $K(X)$, which is said to be an \textbf{ordinary augmentation} of $w$. We will denote it as $w^\prime = [w; f, \b]$.
\end{Definition}

\begin{Remark}\label{Remark gw = g(b)v}
	Take a pair of definition $(b,\g)$ for $w$ over $K$. For any polynomial $g(X)\in K[X]$, write $g(X) = (X-z_1)\dotsc (X-z_n)$ where $z_i\in\overline{K}$. Since $w(X-z_i)\leq v(X-b)=\g$, it follows that $wg\leq vg(b)$. Equality holds whenever $j(g)=0$. Moreover, if $wg=0$ and $j(g)=0$, it follows from [\ref{Dutta invariant of valn tr extns and connection with key pols}, Theorem 3.7] that $gw = g(b)v$.
\end{Remark}

\begin{Lemma}\label{Lemma wK(X) = vK(b)}
		Let $(K(X)|K,w)$ be a valuation transcendental extension. Take an extension of $w$ to $\overline{K}(X)$ and take a minimal pair of definition $(a,\g)$ for $w$ over $K$. Take the minimal polynomial $Q(X)$ of $a$ over $K$ and any key polynomial $f(X)$ of $w$ over $K$. Assume that $\deg f > \deg Q$. Then there is a root $b$ of $f$ such that $wg = vg(b)$ for all $g(X)\in K[X]$ with $\deg g < \deg f$, and $wK(X) = vK(b)$. 
\end{Lemma}

\begin{proof}
	Since $\deg f > \deg Q$, it follows from [\ref{Dutta invariant of valn tr extns and connection with key pols}, Theorem 3.10(3)] that $\g\in v\overline{K}$. Take some $\b\in v\overline{K}$ such that $\b>wf$ and consider the augmentation $w^\prime:= [w; f, \b]$. Take a common extension of $w^\prime$ and $w|_{\overline{K}}$ to $\overline{K}(X)$ which we again denote by $w^\prime$. It then follows from [\ref{Dutta invariant of valn tr extns and connection with key pols}, Lemma 5.4] that there exists a root $b$ of $f$ such that $(b,\g^\prime)$ is a minimal pair of definition for $w^\prime$ over $K$. Hence $j_{w^\prime}(g)=0$ whenever $\deg g < \deg f$. 
	
	\pars Take any $g\in K[X]$ such that $\deg g < \deg f$. Then $w^\prime g = wg$ by definition. Again, $w^\prime g = vg(b)$ by Remark \ref{Remark gw = g(b)v}. We thus have the first assertion. In particular, $wQ = vQ(b)$. As a consequence, $vK(b)\subseteq wK(X)$. Now take any $h\in K[X]$ and write $h = \sum_{i=0}^{n} h_iQ^i$ where $\deg h_i < \deg Q$. Then $wh = \{\min w(h_iQ^i)\}$ by [\ref{Novacoski key poly and min pairs}, Theorem 1.1]. Since $w(h_i Q^i) = v(h_iQ^i(b))$ by our prior observations, we conclude that $wh \in vK(b)$.  
\end{proof}

\begin{Proposition}\label{Prop d(K(a)|K) = d(K(b)|K)}
	Let $(K,v)$ be henselian and $(K(X)|K,w)$ a valuation transcendental extension. Take an extension of $w$ to $\overline{K}(X)$, a minimal pair of definition $(a,\g)$ for $w$ over $K$ and a pair of definition $(b,\g)$. Take the minimal polynomials $Q$ and $f$ of $a$ and $b$ over $K$. Assume that $f$ is a key polynomial for $w$ over $K$. Then
	\[  (vK(b):vK(a))[K(b)v:K(a)v] = \dfrac{\deg f}{\deg Q}.  \]
\end{Proposition}

\begin{Remark}
	With the assumptions of Proposition \ref{Prop d(K(a)|K) = d(K(b)|K)}, it has been observed in [\ref{Dutta imp const fields key pols valn alg extns}, Lemma 3.2] that $vK(a)\subseteq vK(b)$ and $K(a)v\subseteq K(b)v$. Hence the formulation of the statement of the proposition is justified. 
\end{Remark}

\begin{proof}
	If $\deg f = \deg Q$ then $(b,\g)$ is also a minimal pair of definition for $w$ over $K$. The assertion is now immediate in view of [\ref{APZ characterization of residual trans extns}, Theorem 2.1]. We thus assume that $\deg f > \deg Q$. By Lemma \ref{Lemma wK(X) = vK(b)} there exists a $K$-conjugate $b^\prime$ of $b$ such that $wg = vg(b^\prime)$ whenever $\deg g < \deg f$ and $wK(X) = vK(b^\prime)$. Since $(K,v)$ is henselian, we conclude that
	\[ wg = vg(b) \text{ whenever $g(X)\in K[X]$ with } \deg g < \deg f, \text{ and } wK(X) = vK(b).   \]
	Set $e$ to be the order of $wQ$ modulo $vK(a)$. Then 
	\begin{equation}\label{Eqn 13}
		e = (wK(X):vK(a)) = (vK(b):vK(a))
	\end{equation}
	by [\ref{APZ characterization of residual trans extns}, Theorem 2.1]. Take $g\in K[X]$ with $\deg g < \deg Q$ such that $w(gQ^e) = 0$. Write $f = \sum_{i\in S} f_iQ^i$ where $S\subset \NN$ is a finite indexing set and $f_i\in K[X]$ with $\deg f_i <\deg Q$. Since $f$ is a key polynomial for $w$ over $K$, in view of [\ref{Dutta invariant of valn tr extns and connection with key pols}, Remark 3.8] we can assume that $\{0\}\subsetneq S$ and if $\{n\} = \max S$, then $f_n = 1$ and $n = \deg f / \deg Q$. Moreover, $wf = wf_0 = w(f_i Q^i)$ for all $i\in S$. Employing Remark \ref{Remark gw = g(b)v} we observe that $vf_0(a) = w(f_i Q^i)$ for all $i\in S$. As a consequence, $e$ divides $i$ for all $i\in S$. Write
	\[ i = e m_i \text{ for all } i\in S.  \]
	We can thus express
	\[  f = \sum_{i\in S}  \dfrac{f_i}{g^{m_i}} (gQ^e)^{m_i}.   \]
	Take $f^\prime\in K[X]$ with $\deg f^\prime<\deg Q$ such that $wf^\prime = vf^\prime(a) = -vf_0 (a) = -wf_0$. Modify the above expression as
	\begin{equation}\label{Eqn 14}
		f^\prime f = \sum_{i\in S}  \dfrac{f^\prime f_i}{g^{m_i}} (gQ^e)^{m_i}.
	\end{equation} 
	Observe that $w(f^\prime f_i / g^{m_i}) = 0$ for all $i\in S$. Taking residues, we then obtain that
	\begin{equation}\label{Eqn 15}
		f^\prime fw  = \sum_{i\in S}  \dfrac{f^\prime f_i}{g^{m_i}}w (gQ^ew)^{m_i}.
	\end{equation}
	We observe from Remark \ref{Remark gw = g(b)v} that $(f^\prime f_i / g^{m_i})w = (f^\prime(a) f_i(a) / g^{m_i}(a))v\in K(a)v$. Equation (\ref{Eqn 15}) then yields that $f^\prime f w = \chi(gQ^e w)$ where $\chi$ is a polynomial of degree $m_n$ over $K(a)v$.  
	
	\pars Plugging in $b$ in (\ref{Eqn 14}), taking residues and employing Remark \ref{Remark gw = g(b)v}, we obtain that
	\begin{equation}\label{Eqn 16}
		0 = \chi(\zeta), \text{ where } \zeta:= (gQ^e)(b)v. 
	\end{equation}
	Observe that $gQ^ew$ is transcendental over $K(a)v$ by [\ref{APZ characterization of residual trans extns}, Theorem 2.1]. Moreover, $f^\prime f w$ is irreducible over $K(a)v$ by [\ref{Mac Lane key pols}, Lemma 11.2]. Thus $\chi$ is an irreducible polynomial of degree $m_n$ over $K(a)v$. Since $\zeta\in K(b)v$, it now follows from (\ref{Eqn 16}) that 
	\begin{equation}\label{Eqn 17}
		[K(b)v:K(a)v] \geq [K(a)v(\zeta): K(a)v] = m_n.
	\end{equation}
    
    \pars Now take any element $\a\in K(b)v$. Then $\a = h(b)v$ for some $h\in K[X]$ with $\deg h < \deg f$ and $wh=vh(b) = 0$. Write $h = \sum_{i=0}^{d} h_i Q^i$ where $\deg h_i < \deg Q$. Observe that $0 = wh = \min\{w(h_i Q^i)\}$ by [\ref{Novacoski key poly and min pairs}, Theorem 1.1]. Set $S^\prime := \{ i\mid w(h_i Q^i) = 0  \}$. For $i\notin S^\prime$, we then have that $v(h_i Q^i)(b) = w(h_i Q^i) > 0$. Moreover, for $i\in S^\prime$, the fact that $w(h_i Q^i) = 0 \in vK(a)$ implies that $e$ divides $i$. We write 
    \[  i = e t_i \text{ for all } i\in S^\prime.  \]
    It follows that
    \[ h(b)v = \sum_{i\in S^\prime} (h_i Q^i)(b)v = \sum_{i\in S^\prime} \dfrac{h_i}{g^{t_i}}(b) v \zeta^{t_i}.  \]
    By Remark \ref{Remark gw = g(b)v}, we have that $(h_i/g^{t_i})(b)v \in K(a)v$. We have thus shown that $K(b)v \subseteq K(a)v(\zeta)$. From (\ref{Eqn 17}) we conclude that 
    \begin{equation}\label{Eqn 18}
    	[K(b)v:K(a)v] = m_n.
    \end{equation}
    The assertion now follows from (\ref{Eqn 13}), (\ref{Eqn 18}) and the observation that $em_n = n = \deg f / \deg Q$.
\end{proof}

\begin{Remark}\label{Remark defect notion}
	For a unibranched extension $(K(z)|K,v)$ where $z\in \overline{K}$, the value
	\[ d(K(z)|K,v) := \dfrac{[K(z):K]}{(vK(z):vK)[K(z)v:Kv]}  \]
	is referred to as the defect of the extension $(K(z)|K,v)$. The conclusion of Proposition \ref{Prop d(K(a)|K) = d(K(b)|K)} can then also be restated as 
	\[ d(K(a)|K,v) = d(K(b)|K,v).  \]
	Note that a unibranched extension $(K(z)|K,v)$ is defectless if and only if $d(K(z)|K,v)=1$. The Lemma of Ostrowski states that $(K(z)|K,v)$ is always defectless whenever $\ch Kv=0$. On the other hand if $\ch Kv = p>0$ then $d(K(z)|K,v) = p^n$ for some $n\in\NN$. 
\end{Remark}

\subsection{Proof of Theorem \ref{Thm K(b,X)|K(b) is defectless}}

\begin{proof}
	Take an extension of $w$ to $\overline{K(X)}$. In light of Proposition \ref{Prop j(f) = j(f^h)} we can assume that $(K,v)$ is henselian. Then
	\begin{equation}\label{Eqn 19}
		[K(b,X)^h:K(X)^h] = j(f)
	\end{equation}
by Corollary \ref{Coro j(f) = [K(b):IC_K] when K=K^h}. Take a minimal pair of definition $(a,\g)$ for $w$ over $K$ and take the minimal polynomial $Q(X)$ of $a$ over $K$. Write 
\[ Q(X) = (X-a_1)\dotsc(X-a_n),  \]
where we identify $a_1$ with $a$, and the roots are indexed such that $v(a-a_i)\geq \g$ for all $1\leq i\leq j$ and $v(a-a_i)<\g$ otherwise. Thus $j = j(Q)$. It follows that 
\[ wQ = j\g + \a, \]
where $\a = \sum_{i=j+1}^{n} v(a-a_i) \in vK(a)$ by [\ref{APZ characterization of residual trans extns}, Theorem 2.1]. Further, we observe from [\ref{Dutta imp const fields key pols valn alg extns}, Lemma 3.2] that
\[ vK(a)\subseteq vK(b) \text{ and } K(a)v\subseteq K(b)v. \]

\pars We first assume that $\g\in v\overline{K}$. Set $E$ to be the order of $\g$ modulo $vK(b)$ and $e$ to be the order of $wQ$ modulo $vK(a)$. Then $e$ is also the order of $j\g$ modulo $vK(a)$. It follows from [\ref{APZ characterization of residual trans extns}, Theorem 2.1] that $(wK(b,X): vK(b)) = E$ and $(wK(X): vK(a)) = e$. Consequently, 
\begin{equation}\label{Eqn 20}
	(wK(b,X): wK(X)) = \dfrac{E (vK(b): vK(a))}{e}.
\end{equation}
We further observe from [\ref{APZ characterization of residual trans extns}, Theorem 2.1] that 
\[ K(b,X)w = K(b)v (d(X-b)^Ew) \text{ and }  K(X)w = K(a)v (gQ^ew),    \]
where $d\in K(b)$ such that $vd = -E\g$ and $g(X)\in K[X]$ with $\deg g < \deg Q$ such that $wg = -ewQ$. The fact that $ewQ\in vK(a)\subseteq vK(b)$ implies that $ej\g\in vK(b)$. Hence $E$ divides $ej$. Set
\[ f:= \dfrac{ej}{E}.   \]
Then $vd^f = -ej\g$. Take $c\in K(b)$ such that $vc = -e\a$. Thus $vcd^f = wg$. It follows from [\ref{APZ2 minimal pairs}, Proposition 1.1] that $\dfrac{cd^f}{g} w$ is algebraic over $Kv$. Since $\dfrac{cd^f}{g}\in K(b,X)$, we conclude that $\dfrac{cd^f}{g} w \in K(b)v$. Thus
\[ K(b)v (gQ^e w) = K(b)v (cd^f Q^e w).   \] 
We thus have the chain of containments: 
\begin{equation}\label{Eqn 21}
	K(X)w = K(a)v (gQ^ew) \subseteq K(b)v (cd^f Q^e w) \subseteq K(b)v (d(X-a)^Ew) = K(b,X)w.
\end{equation}
Observe that $gQ^ew$ is transcendental over $K(a)v$ by [\ref{APZ characterization of residual trans extns}, Theorem 2.1]. Hence,
\begin{equation}\label{Eqn 22}
	[K(b)v (cd^f Q^e w): K(X)w] = [K(b)v (g Q^e w) : K(a)v (gQ^ew)] = [K(b)v: K(a)v].
\end{equation}
We now express 
\[ cd^f Q^e = c_0 + c_1 (X-b) + \dotsc + c_m (X-b)^m \text{ where } c_i \in K(b).   \]
The facts that $w(cd^fQ^e) = 0$ and $(b,\g)$ is a pair of definition for $w$ imply that $w(c_i (X-b)^i) \geq 0$ for all $i$. The minimality of $E$ then implies that 
\[ w(c_i (X-b)^i) > 0 \text{ whenever } E \text{ does not divide }i. \]
As a consequence, 
\begin{align*}
	cd^fQ^e w &= c_0 v + c_E (X-b)^Ew + \dotsc + c_{mE}(X-b)^{mE}w\\
	&= c_0 v + (\dfrac{c_E}{d}v) d(X-b)^E w + \dotsc + (\dfrac{c_{mE}}{d^m}v) (d(X-b)^Ew)^m \in K(b)v [d(X-b)^E w].
\end{align*}
We have now expressed $cd^fQ^e w$ as a polynomial $\chi$ in the variable $d(X-b)^E w$ over $K(b)v$. Thus, 
\begin{equation}\label{Eqn 23}
	[K(b,X)w: K(b)v (cd^fQ^e w)] = \deg \chi.
\end{equation}
Observe that $c_{fE} = c_{ej}$ is the coefficient of $(X-b)^{ej}$ in $cd^f Q^e$. Thus 
\[ c_{ej} = cd^f(-1)^{ne-je} \mathcal{E}_{ne-je} (a_1-b, \dotsc , a_n -b),   \]
where each $a_i-b$ appears $e$ times and $\mathcal{E}_k$ is the $k$-th elementary symmetric polynomial. There is a unique contributing factor of the smallest value in $\mathcal{E}_{ne-je}(a_1-b, \dotsc, a_n-b)$, namely $(a_{j+1}-b)^e \dotsc (a_n-b)^e$. It follows from the triangle inequality that
\[ vc_{fE} = vc_{ej} = v(cd^f) + v((a_{j+1}-b)^e \dotsc (a_n-b)^e) = -ewQ + e\a = -ej\g = -fE\g = vd^f.  \]
Thus $\dfrac{c_{fE}}{d^f}v \neq 0$. As a consequence, 
\begin{equation}\label{Eqn 24}
	\deg \chi \geq f.
\end{equation}
We now take some $i>f$. Then $ne - iE < ne - je$. Observe that
\[ c_{iE} = cd^f (-1)^{ne - iE} \mathcal{E}_{ne-iE}(a_1 - b, \dotsc , a_n-b).  \] 
Take any contributing factor $cd^f (a_{t_1} - b) \dotsc (a_{t_{ne-iE}}-b)$ of $c_{iE}$. The fact that $ne-iE < ne-je$ implies that
\[ v((a_{t_1} - b) \dotsc (a_{t_{ne-iE}}-b)) + (iE-ej)\g > v((a_{j+1}-b)^e \dotsc (a_n-b)^e).   \]
As a consequence, $vc_{iE} + (iE-ej)\g > vcd^f +v((a_{j+1}-b)^e \dotsc (a_n-b)^e)  = -ej\g$ and hence
\[ vc_{iE} > -iE\g = vd^i. \]
We have thus shown that
\[ \dfrac{c_{iE}}{d^i}v = 0 \text{ whenever } i>f.  \]
Hence $\deg\chi = f$ from (\ref{Eqn 24}). From (\ref{Eqn 21}), (\ref{Eqn 22}) and (\ref{Eqn 23}) we conclude that
\[ [K(b,X)w: K(x)w] = f[K(b)v:K(a)v].   \]
Combining with (\ref{Eqn 20}) we obtain that
\begin{equation}\label{Eqn 25}
	(wK(b,X):wK(X))[K(b,X)w:K(X)w] = j (vK(b): vK(a))[K(b)v:K(a)v]. 
\end{equation}

\pars We now assume that $\g\notin v\overline{K}$. It follows from [\ref{Dutta min fields implicit const fields}, Remark 3.3] that $wK(b,X) = vK(b)\dirsum\ZZ\g$ and $wK(X) = vK(a) \dirsum \ZZ j\g $. Hence $(wK(b,X) : wK(X)) = j(vK(b): vK(a))$. Moreover, $[K(b,X)w:K(X)w] = [K(b)v:K(a)v]$. We thus again obtain the following relations in this case: 
\begin{equation}\label{Eqn 26}
	(wK(b,X):wK(X))[K(b,X)w:K(X)w] = j (vK(b): vK(a))[K(b)v:K(a)v]. 
\end{equation}

\pars We first assume that $(K(b)|K,v)$ is defectless. By Theorem \ref{Thm cdc Khanduja-Aghigh} we can take a complete distinguished chain $b,b_1, \dotsc , b_n$ of $b$ over $K$. If $\g> \d(b,K)$ then $(b,\g)$ is a minimal pair of definition for $w$ over $K$. Else $v(b-b_1)\geq \g$ and hence $(b_1, \g)$ is also a pair of definition for $w$. Since $b_n\in K$, repeated application of this observation would yield some $b_i$ in this chain such that $(b_i, \g)$ is a minimal pair of definition for $w$ over $K$. We can thus set $a = b_i$ without any loss of generality. Since $a$ lies in the complete distinguished chain $b,b_1, \dotsc , b_n$, it follows from Theorem \ref{Thm cdc Khanduja-Aghigh} that $(K(a)|K,v)$ is defectless as well. Equations (\ref{Eqn 25}) and (\ref{Eqn 26}) can now be modified as 
\begin{equation}\label{Eqn 27}
	(wK(b,X):wK(X))[K(b,X)w:K(X)w] = j \dfrac{\deg f}{\deg Q} = j(f),  
\end{equation}
where the last equality follows from Lemma \ref{Lemma deg f/ deg Q = j(f)/j(Q)}. Equation (\ref{Eqn 27}) also holds in the case when $f$ is a key polynomial for $w$ over $K$ by Proposition \ref{Prop d(K(a)|K) = d(K(b)|K)}. The theorem now follows from Equations (\ref{Eqn 19}) and (\ref{Eqn 27}). 
\end{proof}

\pars The conclusion of Theorem \ref{Thm K(b,X)|K(b) is defectless} fails to hold if we remove the assumptions, as evidenced by the next example. 

\begin{Example}\label{Example K(b,X)|K(X) defect}
	Let $k$ be an imperfect field endowed with a non-trivial valuation $v$. Set $K$ to be the separable-algebraic closure of $k$ and take an extension of $v$ to $\overline{K}$. Then $(\overline{K}|K,v)$ is an immediate extension. Take $b\in \overline{K}\setminus K$. Then $(K(b)|K,v)$ is a non-trivial immediate unibranched extension and hence is a defect extension. Take the extension $w:= v_{0,0}$ and take an extension of $w$ to $\overline{K(X)}$. We can assume that $vb > 0$ without any loss of generality. Thus $(b, 0)$ is also a pair of definition. It follows from [\ref{APZ characterization of residual trans extns}, Theorem 2.1] that $wK(b,X) = vK(b) = vK = wK(X)$ and $K(b,X)w = K(b)v (Xw) = Kv (Xw) = K(X)w$. Thus $(K(b,X)^h|K(X)^h, w)$ is also immediate. Now henselization being a separable extension is linearly disjoint to a purely inseparable extension and hence,
	\[ [K(b,X)^h : K(X)^h] = [K(b):K] > 1.   \]
	It follows that $(K(b,X)|K(X),w)$ is a defect extension.   
\end{Example}

\pars This next example illustrates that the converse to both Proposition \ref{Prop d(K(a)|K) = d(K(b)|K)} and Theorem \ref{Thm K(b,X)|K(b) is defectless} fail to hold. We use notations as in Remark \ref{Remark defect notion}.

\begin{Example}\label{Example K(b,X)|K(X) defectless}
	Take an odd prime $p$ and denote by $(k,v)$ the valued field $\FF_p(t)$ equipped with the $t$-adic valuation. Fix an extension of $v$ to $\overline{k}$. Denote by $K$ the henselization of the perfect closure of $k$. Take a root $a$ of the Artin-Schreier polynomial $Q(X):= X^p-X-1/t \in K[X]$. Then $d(K(a)|K,v) = [K(a):K] = p$ (see [\ref{Kuh defect}, Example 3.9] for a proof). Take $\g\in\RR\setminus\QQ$ such that $0<\g<1/2$ and set $w:= v_{a,\g}$. Observe that every $K$-conjugate $a^\prime$ of $a$ is of the form $a+i$ for some $i\in\FF_p$ and hence $v(a-a^\prime) = 0$. Thus for any pair of definition $(z,\g)$ of $w$, we have that
	\begin{equation}\label{Eqn 28}
		K(a)\subseteq K(z)
	\end{equation} 
	 as a consequence of a variant of the Krasner's Lemma [\ref{Kuh value groups residue fields rational fn fields}, Lemma 2.21]. In particular, $(a,\g)$ is a minimal pair of definition for $w$ over $K$.
	 
	 \pars Set $b:= a+\sqrt{t}$. Then $v(a-b) = 1/2>\g$ and hence $(b,\g)$ is a pair of definition for $w$. Employing (\ref{Eqn 28}) we observe that $K(b) = K(a,\sqrt{t})$. It follows that $[K(b):K] = 2p$. Since defect is multiplicative and is always a power of $p$, we conclude that
	 \[ d(K(a)|K,v) = d(K(b)|K,v) = p.   \]
	 Take the minimal polynomial $f$ of $b$ over $K$. Observe that $j(Q) = 1$ by construction. Then $j(f) = 2$ by Lemma \ref{Lemma deg f/ deg Q = j(f)/j(Q)}. Consequently, $[K(b,X)^h: K(X)^h] = 2$ by Corollary \ref{Coro j(f) = [K(b):IC_K] when K=K^h}. It follows that
	 \[ (K(b,X)|K(X),w) \text{ is defectless}.   \]
	 Finally, the fact that $\g\notin\QQ$ implies that $\g$ is torsion-free modulo $vK$. Since $\deg f > \deg Q$, it now follows from [\ref{Dutta invariant of valn tr extns and connection with key pols}, Theorem 3.10(3)] that $f$ is not a key polynomial of $w$ over $K$.
\end{Example}


\begin{thebibliography}{1000000000}
	\bibitem{KASKK} K. Aghigh and S. K. Khanduja, On chains associated with elements algebraic over a henselian valued field, Algebra Colloq., 12(4) (2005), 607--616. \label{KA SKK chains associated with elts henselian}	
	\bibitem{AP88} V. Alexandru and A. Zaharescu, Sur une classe de prolongements \`{a} $K(x)$ d'une valuation sur une corp $K$, Rev. Roumaine Math. Pures Appl., 5 (1988), 393--400. \label{AP sur une classe}  	
	\bibitem{APZ1} V. Alexandru, N. Popescu and A. Zaharescu, A theorem of characterization of residual transcendental extensions of a valuation, J. Math. Kyoto University, 28 (1988), 579--592. \label{APZ characterization of residual trans extns}
	\bibitem{APZ2} V. Alexandru, N. Popescu and A. Zaharescu, Minimal pairs of definition of a residual transcendental extension of a valuation, J. Math. Kyoto University, 30 (1990), 207--225. \label{APZ2 minimal pairs}
	\bibitem{ABFVK17} A. Blaszczok and F.-V. Kuhlmann, On maximal immediate extensions of valued fields, Mathematische Nachrichten, 290 (2017), 7--18. \label{Kuh max imm extns of valued fields}
	\bibitem{DMS 18} J. Decaup, W. Mahboub and M. Spivakovsky, Abstract key polynomials and comparison theorems with the key polynomials of Mac Lane--Vaqui\'{e}, Illinois Journal of Mathematics 62(1--4), 253--270 (2018). \label{Decaup Spiva Mahboub ABKP comparison MVKP} 
    \bibitem{Du24} A. Dutta, An invariant of valuation transcendental extensions and its connection with key polynomials, Journal of Algebra, 649 (2024), 133--168. \label{Dutta invariant of valn tr extns and connection with key pols}
	\bibitem{Du21} A. Dutta, Minimal pairs, minimal fields and implicit constant fields, Journal of Algebra, 588 (2021), 479--514. \label{Dutta min fields implicit const fields}
	\bibitem{Du21} A. Dutta, Minimal pairs, inertia degrees, ramification degrees and implicit constant fields, Communications in Algebra, 50(11) (2022), 4964--4974. \label{Dutta min pairs inertia ram deg impl const fields}
	\bibitem{Du21} A. Dutta, On the implicit constant fields and key polynomials for valuation algebraic extensions, Journal of Commutative Algebra, 14(4) (2022), 515--525. \label{Dutta imp const fields key pols valn alg extns}
	\bibitem{K2} F.-V. Kuhlmann, The defect, in Commutative Algebra - Noetherian and non-Noetherian perspectives. Marco Fontana, Salah-Eddine Kabbaj, Bruce Olberding and Irena Swanson eds. Springer, 2011. \label{Kuh defect}	
	\bibitem{K1} F.-V. Kuhlmann,   Valuation theoretic and model theoretic aspects of local uniformization, in Resolution of Singularities -
	A Research Textbook in Tribute to Oscar Zariski, H. Hauser, J. Lipman, F. Oort, A. Quiros (es.), Progress in Math. 181, Birkh\"auser (2000), 4559-4600. \label{Kuh vln model}
	\bibitem{FVK-04} F.-V. Kuhlmann, Value groups, residue fields and bad places of rational function fields, Trans. Amer. Math. Soc., 356 (2004), 4559--4600. \label{Kuh value groups residue fields rational fn fields}
	\bibitem{K8} F.-V. Kuhlmann, Elimination of Ramification I: The Generalized Stability Theorem, Trans. Amer. Math. Soc., 362(11) (2010), 5697--5727. \label{Kuh gen stab thm} 
	\bibitem{SM36} S. Mac Lane, A construction for absolute values in polynomial rings, Trans. Amer. Math. Soc., 40 (1936), 363--395. \label{Mac Lane key pols}
	\bibitem{SM36} S. Mac Lane, A construction for prime ideals as absolute values of an algebraic field, Duke Math. J. (1936), 492--510. \label{Mac Lane key pols prime ideals}
    \bibitem{EN-JN-23} E. Nart and J. Novacoski, The defect formula, Advances in Mathematics, 428 (2023), DOI: \url{https://doi.org/10.1016/j.aim.2023.109153} \label{Nart Novacoski defect formula}
	\bibitem{JN-19} J. Novacoski, Key polynomials and minimal pairs, Journal of Algebra, 523 (2019), 1--14. \label{Novacoski key poly and min pairs}
	\bibitem{JN-MS-18}	J. Novacoski and M. Spivakovsky, Key polynomials and pseudo-convergent sequences, Journal of Algebra, 495 (2018), 199--219. \label{Nova Spiva key pol pseudo convergent}
	\bibitem{MV07} M. Vaqui\'{e}, Extension d'une valuation, Trans. Amer. Math. Soc., 359 no. 7 (2007), 3439--3481. \label{Vaquie key pols} 
\end{thebibliography}
\end{document}